\documentclass[letterpaper, 10 pt, conference, final]{ieeeconf}  

\IEEEoverridecommandlockouts                              

\usepackage{fix-cm}
\usepackage{etex}

\usepackage{dblfloatfix}

\usepackage{nag}

\makeatletter
\@ifpackageloaded{xcolor}{}{%
\usepackage[table,x11names,dvipsnames,svgnames]{xcolor}%
}
\makeatother

\usepackage{colortbl}

\usepackage{graphicx}
\usepackage{wrapfig}

\definecolorset{RGB}{lyft}{}{Red,194,39,36;Sunset,202,53,33;Orange,205,68,20;Amber,200,117,42;Yellow,242,169,52;Citron,186,188,44;Lime,112,159,33;Green,56,139,31;Mint,45,118,56;Teal,52,133,135;Cyan,60,132,202;Blue,55,94,248;Indigo,64,13,247;Purple,115,42,248;Pink,176,25,145;Rose,176,32,75}

\usepackage{cite}

\usepackage{microtype}

\usepackage[american]{babel}

\usepackage{array}
\usepackage{multirow}
\usepackage{booktabs}
\usepackage{makecell} 

\ifcsname labelindent\endcsname
\let\labelindent\relax
\fi
\usepackage[inline]{enumitem}

\usepackage{subfig}

\newenvironment{lenumerate}[2][]
{\begin{enumerate}[label=(#2\arabic*),leftmargin=0.2in,itemindent=0.15in,#1]}
{\end{enumerate}}

\setlist*[enumerate,1]{label={\itshape\arabic*)}}

\makeatletter
\newcommand{\paragraphswithstop}{%
\let\copyparagraph\paragraph%
\renewcommand\paragraph[1]{\copyparagraph{##1.}}%
}
\makeatother

\usepackage[framemethod=tikz]{mdframed}

\makeatletter
\def\namedlabel#1#2{\begingroup
  #2%
  \def\@currentlabel{#2}%
  \phantomsection\label{#1}\endgroup
}
\makeatother

\usepackage{suffix}

\usepackage{environ}

\makeatletter
\newsavebox{\boxifnotempty}
\newcommand{\displayifnotempty}[3]{\sbox\boxifnotempty{#2}\setbox0=\hbox{\usebox{\boxifnotempty}\unskip}%
\ifdim\wd0=0pt
\else
 #1\usebox{\boxifnotempty}#3%
\fi%
}
\newcommand{\ifempty}[2]{\setbox0=\hbox{#1\unskip}%
\ifdim\wd0=0pt%
 #2%
\fi%
}
\newcommand{\ifnotempty}[2]{\setbox0=\hbox{#1\unskip}%
\ifdim\wd0>0pt%
 #2%
\fi%
}
\makeatother

\usepackage{algorithm}

\usepackage{scrlfile}

\makeatletter
\newcommand*\newstoreddef[1]{
  \BeforeClosingMainAux{%
    \immediate\write\@auxout{%
      \string\restoredef{#1}{\csname #1\endcsname}%
    }%
  }%
}
\newcommand*{\restoredef}[2]{
  \expandafter\gdef\csname stored@#1\endcsname{#2}%
}
\newcommand*{\storeddef}[1]{
  \@ifundefined{stored@#1}{0}{\csname stored@#1\endcsname}%
}
\makeatother

\usepackage{pageslts}
\NewEnviron{tee}{\BODY\typeout{Marker Tee [start] ^^J \BODY ^^JMaker Tee [end]}}

\input{preamble/math}
\newcommand{\real}[1]{\mathbb{R}^{#1}{}}

\newcommand{\bmat}[1]{\begin{bmatrix}#1\end{bmatrix}}

\newcommand{\transpose}{^\mathrm{T}}

\DeclarePairedDelimiter{\norm}{\lVert}{\rVert}

\DeclareMathOperator{\expect}{\mathbb{E}}

\providecommand{\cC}{\mathcal{C}}

\providecommand{\cP}{\mathcal{P}}

\providecommand{\cU}{\mathcal{U}}

\providecommand{\cX}{\mathcal{X}}

\usepackage{units}

  \newcommand{\newcolorlabel}[2]{%
  \expandafter\newcommand\csname #1\endcsname[1]{%
    \tikz[baseline]{\node[text=white,fill=#2,anchor=base,text height=1.3ex,text depth=0.1ex,font=\sffamily\bfseries]{##1}}}%
}

\newcommand{\newcommenter}[2]{%
  \expandafter\newcommand\csname #1\endcsname[1]{%
    \fcolorbox{#2}{#2}{\color{white}\textsf{\textbf{#1}}}
    {\color{#2}##1}}%
  \expandafter\newcommand\csname at#1\endcsname{%
    \fcolorbox{#2}{#2}{\color{white}\textsf{\textbf{@#1}}}
    {\color{#2}}}%
  \expandafter\newcommand\csname #1hl\endcsname[2]{%
    \colorbox{#2}{\color{white}\textsf{\textbf{#1}}}\sethlcolor{Azure2}\hl{##2}~%
    \expandafter\ifx\csname commentarrow\endcsname\relax$\leftarrow$\else \commentarrow[#2]\fi~%
    {\color{#2}##1}}%
  \expandafter\newcommand\csname #1st\endcsname[2]{%
    \colorbox{#2}{\color{white}\textsf{\textbf{#1}}}\sout{##2}~%
    \expandafter\ifx\csname commentarrow\endcsname\relax$\leftarrow$\else \commentarrow[#2]\fi~%
    {\color{#2}##1}}%
}
\newcommenter{TODO}{DodgerBlue1}
\newcommenter{rtron}{Green3}

\usepackage{comment}

\usepackage{pdfcomment}

\usepackage{soul}

\usepackage[normalem]{ulem}

\usepackage{csquotes}

\usepackage{tikz}
\usetikzlibrary{calc}
\usetikzlibrary{matrix}
\usetikzlibrary{chains}
\usetikzlibrary{shapes.geometric}
\usetikzlibrary{arrows.meta}
\usetikzlibrary{decorations.pathreplacing}
\usetikzlibrary{backgrounds}

\tikzset{
  dim above/.style={to path={\pgfextra{
        \pgfinterruptpath
        \draw[>=latex,|->|] let
        \p1=($(\tikztostart)!1.5em!90:(\tikztotarget)$),
        \p2=($(\tikztotarget)!1.5em!-90:(\tikztostart)$)
        in(\p1) -- (\p2) node[pos=.5,sloped,above]{#1};
        \endpgfinterruptpath
      }
    }
  },
  dim double above/.style={to path={\pgfextra{
        \pgfinterruptpath
        \draw[>=latex,|->|] let
        \p1=($(\tikztostart)!3em!90:(\tikztotarget)$),
        \p2=($(\tikztotarget)!3em!-90:(\tikztostart)$)
        in(\p1) -- (\p2) node[pos=.5,sloped,above]{#1};
        \endpgfinterruptpath
      }
    }
  },
  dim below/.style={to path={\pgfextra{
        \pgfinterruptpath
        \draw[>=latex,|->|] let 
        \p1=($(\tikztostart)!-1em!-90:(\tikztotarget)$),
        \p2=($(\tikztotarget)!-1em!90:(\tikztostart)$)
        in (\p1) -- (\p2) node[pos=.5,sloped,below]{#1};
        \endpgfinterruptpath
      }
    }
  },
}

\tikzset{
    right angle quadrant/.code={
        \pgfmathsetmacro\quadranta{{1,1,-1,-1}[#1-1]}     
        \pgfmathsetmacro\quadrantb{{1,-1,-1,1}[#1-1]}},
    right angle quadrant=1, 
    right angle length/.code={\def\rightanglelength{#1}},   
    right angle length=2ex, 
    right angle symbol/.style n args={3}{
        insert path={
            let \p0 = ($(#1)!(#3)!(#2)$) in     
                let \p1 = ($(\p0)!\quadranta*\rightanglelength!(#3)$), 
                \p2 = ($(\p0)!\quadrantb*\rightanglelength!(#2)$) in 
                let \p3 = ($(\p1)+(\p2)-(\p0)$) in  
            (\p1) -- (\p3) -- (\p2)
        }
    }
}

\newcommand{\pgfextractangle}[3]{%
    \pgfmathanglebetweenpoints{\pgfpointanchor{#2}{center}}
                              {\pgfpointanchor{#3}{center}}
    \global\let#1\pgfmathresult  
}

\usetikzlibrary{shapes.arrows}
\newcommand{\commentarrow}[1][Azure4]{\tikz[baseline=-3pt]{\node[shape border uses incircle, fill=#1,rotate=180,single arrow, inner sep=1pt, minimum size=6pt, single arrow head extend=2pt]{};}}

\tikzset{ax/.style={-latex,line width=2pt}}

\tikzset{camera/.style={fill=Sienna1,fill opacity=0.5},%
image plane/.style={draw=RoyalBlue3,line width=2pt}}

\usepackage{balance}

\title{\LARGE \bf
Chance Constraint Robust Control with Control Barrier Functions}

\author{Chenfei Wang$^1$, Mahroo Bahreinian$^2$  and Roberto Tron$^3$
\thanks{$^{1}$Chenfei Wang is with Department of Mechanical Engineering,
        Boston, MA, 02215 USA. Email:
        {\tt\small wang1029@bu.edu}}%
\thanks{$^{2}$Mahroo Bahreinian is with Division of Systems Engineering at Boston
University, Boston, MA, 02215 USA. Email:
        {\tt\small mahroobh@bu.edu}}%
\thanks{$^{3}$Roberto Tron is with Faculty of Department of Mechanical Engineering
at Boston University, Boston, MA, 02215 USA. Email:
        {\tt\small tron@bu.edu}}%
}

\begin{document}

 \maketitle
\thispagestyle{empty}
\pagestyle{empty}

\begin{abstract}

  In this paper, we propose a novel approach to synthesize linear feedback controllers for navigating in polygonal environments using noisy measurements and a convex cell decomposition. Our method is based on formulating chance constraints for the convergence and collision avoidance condition. In particular, the stability and safety guarantees come from chance Control Barrier Function (CBF) constraints and chance Control Lyapunov Function (CLF) constraints, respectively. We use convex over-approximations to get upper bounds of the constraints, leading to a convex robust quadratic program for finding the controller. We apply and provide simulation results for equilibrium control and path control. The result shows that the controller is robust with the noise input.

\end{abstract}

\section{INTRODUCTION}
Safety is a critical requirement in practical applications of control systems. An instance of this problem can be described as finding a controller such that the trajectory of the closed-loop system converges to a desired set (e.g., an equilibrium point) while respecting the safety limits (e.g., the obstacle's boundaries). For example, a driverless car might need to negotiate turns in an intersection without hitting the sidewalk, an industrial manipulator might need to move pieces without striking other equipment, and a robotic vacuum cleaner might need to move in an apartment without touching the furniture. Most of existing methods focus on designing nominal trajectories and controllers in the environment using deterministic measurements, or even full state information of the agent (e.g., the precise position). However, in real situations, measurements are always corrupted by noise. This means that a purely deterministic approach might fail to translate theoretical guarantees into practice. In contrast, our paper presents a method to design linear state-feedback controllers while considering the noisy environment directly together the safety requirements. We define a chance-constrained version of Control Barrier Function constraints (\emph{chance CBF}) and provide sufficient conditions for almost-sure forward invariance of a set. This allow us to set up a framework for automatically designing controllers with noisy input. Although the chance CBF constraints are non-convex, we show how they can be over-approximated to transform the problem into a quadratic problem that can be easily handled by modern solvers. We combine the CBF constraints with Control Lyapunov Function (CLF) constraints to apply our method to two stability objectives: equilibrium control, and path control.

\subsection{Background and previous work}
Barrier function methods have seen a recent rise in popularity, due to their natural flexibility given by their relationship with Lyapunov functions \cite{i1}, and their ability to guarantee safety \cite{i2}. The natural extension of a barrier function to a system with control inputs
is a Control Barrier Function, first proposed in \cite{i3}. The seminal works \cite{c1} and \cite{c1b} summarize their development, and describe the construction of the quadratic programs (QPs) that are at the core of these methods. Their application has been extended to systems with high relative degree \cite{c2}, \cite{c3}, and to the case where stability and safety constraints cannot be satisfied at the same time \cite{m3}. In particular, in this paper we base our treatment of CBF on the work of \cite{c2}.
Additional developments have considered the challenges brought by unknown system disturbances. The work in \cite{m1} considers the case of bounded, state-independent measurements, resulting in a robust formulation of the CBF quadratic program; that approach does not explicitly consider probabilistic chance constraints. The work in \cite{m2} takes a similar approach, where, however, the bounds on the disturbances are obtained from the estimated covariance of the Unscented Kalman Filter (UKF). A major point common to all the aforementioned works is that the implementation of the controller requires solving a QP at every time instant (as opposed to directly synthesizing a output feedback controller as in this paper). 

Chance constrained optimization problems were introduced in \cite{i4}. Even though the it is nonconvex in general, \cite{i5} shows that the feasible set of a chance constraint is convex. \cite{i6}, \cite{c7} and \cite{c8} introduce some methods about convex approximation, including Chebyshev bound which is used in this paper.

The present paper builds upon the work of \cite{c6}, which also consider the problem of learning linear controllers in convex cells. However, that work does not explicitly consider noisy measurements, and the stabilization to single equilibrium points requires the introduction of additional cells in the decomposition (while in this paper we can directly use a quadratic Lyapunov function).

\subsection{Paper contributions}
The first contribution of this paper is the chance CBF theory. We introduce the chance CBF and prove that system satisfying CBF constraint has set invariance in the sense of probability. Equivalently, the probability of system going outside this set is zero. The chance CBF allow us to design the controller satisfying the safety requirement with the noise input. Using chance CBF constraint to set up the controller designing would lead to a nonconvex optimization problem. 

The second contribution of this paper is the method of solving this nonconvex optimization problem. We use the convex approximation to get a upper bound of the probability at each point. Using this upper bound to restrict the probability gives a feasible solution for the original nonconvex problem. To get the control feedback matrix, each point inside the zone needs to satisfying the chance constraint. This is a robust optimization problem. We use the max-min inequality to get the upper bound for this problem in the given set. Also, this process leads to a maximization of a convex function. We use the property of convex to get once more the upper bound to restrict this value. This gives a quadratic constraint quadratic programming (QCQP) in the end.  

The third contribution of this paper is that we validate our approach through numerical simulation. The results show that the controller works very well in the weak noise environment. And the controller also can handle the strong noise to some extend.

\subsection{Organization}
The remainder of the paper is organized as follows. We first review background information and establish basic definitions and results in Section \ref{section2}. We then give the definition of chance CBF and chance CLF constraints, and give sufficient conditions for almost sure forward invariance of a constraint set in Section \ref{section3}. In Section \ref{section4} we discuss our proposed methods for finding controller via quadratically constrained quadratic programming. Finally, in Section \ref{section5} we use simulations to study the performance of our method under challenging conditions. 
 
\section{NOTATION AND PRELIMINARIES} \label{section2}
In this section, we formally define our problem, review relevant background knowledge and definitions, and introduce basic results regarding chance constraints and robust optimization.

\subsection{System Dynamics}
Consider the LTI model,
\begin{equation} \label{LTI}
 \dot{x} =A x+B u
\end{equation}

where $A \in \mathbb{R}^{n_{x} \times n_{x}}$, and $B \in \mathbb{R}^{n_{x} \times n_{u}}$ are matrices defining the linear dynamics, $x\in \mathcal{X}\subset\real{n_x}$ is the state, $u\in \cU\subset\real{n_u}$ is the system input, and we use $\cU$ to model actuator constraints. We assume that the sets $\cX$ and $\cU$ are polyhedra described by sets of linear inequalities $A_xx\leq b_x$ and $A_ux\leq b_u$, respectively, with $A_x\in\real{n_h\times n_x}$, $b_x\in\real{n_h}$, $A_u\in\real{n_{hu}\times n_x}$, $b_u\in\real{n_{hu}}$.

\begin{remark}
For more concise expression, the following paper would use $f$ and $g$ to represent $Ax$ and $B$ respectively.
\end{remark}

\subsection{Problem statement}\label{sec:problem_statement}
We assume that the agent (robot) can measure its own state $x$ corrupted by an additive noise term $\theta\in\real{n_x}$ which is a stochastic variable with mean $\expect[\theta]=0$ and covariance $\expect[\theta\theta\transpose]=\Sigma(x)$, where we assume that $\Sigma(x)\in\real{n_x\times n_x}$ can be a piecewise linear function of $x$.
We then formulate the following:
\begin{problem}\label{def:problem}
  Design a matrix $K\in\real{n_x\times n_x}$ for the linear state feedback controller
  \begin{equation}
u = K(\tilde{x}  + \theta)\label{eq:controller}
\end{equation}
that achieves either of the following objectives (stability) while keeping the state of the system in a polytope $\cX$:
  \begin{lenumerate}{P}
  \item\label{it:equilibrium} \emph{Equilibrium control:} Drive the state toward an equilibrium point $x_e\in\cX$.
  \item\label{it:path} \emph{Path control:} Drive the state toward an exit face of the polytope $\cX$ defined by $z\transpose x+b_x$.
  \end{lenumerate}
\end{problem}
\begin{remark} Although in this paper we focus on individual convex domains, the stabilization objectives \ref{it:equilibrium} and \ref{it:path} can be combined with a convex cell decomposition and a discrete planner to achieve navigation in general polygonal environments (see \cite{c6} and Section~\ref{section5}).
\end{remark}
\subsection{Lie derivatives and relative degree}
Given a sufficiently smooth function $h:\real{n}\to\real{}$ and a vector field $f:\real{n}\to\real{n}$, we use the notation $L_fh=\nabla_x h\transpose f$ to denote the Lie derivative of $h$ along $f$, where $\nabla h$ represents the gradient field of $h$. Higher order derivatives are recursively defined as $L_f^rh=L_fL_f^{r-1}h$, with $L_f^0h=h$.

System \eqref{LTI} is said to have relative degree $R$ with respect to a function $h$ if $L_gL_{Ax}^{i}h=0$ for all $0\leq i<r$ and $L_BL_{Ax}^rh\neq 0$.

\subsection{Control Barrier Function and Safety Constraints}\label{sec:CBFintro}
Consider the system \eqref{LTI} and a constraint set of the form 
\begin{equation} \label{BF}
C=\left\{x \in \mathbb{R}^{n_x} \mid h(x) \geq 0\right\},
\end{equation}
where $h: \mathbb{R}^{n_x} \rightarrow\real{}$ is a sufficiently smooth function (in our application, it will be an affine function). 

Let $r$ be the relative degree of \eqref{LTI} with respect to \eqref{BF}.

Following \cite{c2}, define a series of functions $\psi_{0},\psi_{1},\dots, \psi_{r}$ of the form
\begin{align}\label{eq:psi}
\psi_{0}(x)&=b(x) \\
\psi_{i}(x)&=\dot{\psi}_{i-1}(x)+\alpha_{i}\left(\psi_{i-1}(x)\right) 
\end{align}
and a series of sets $C_{0},C_{1},\dots, C_{r}$
\begin{equation}
C_{i}=\left\{x \in \mathbb{R}^{n_x}: \psi_{i}(x) \geq 0\right\} 
\end{equation}
with $0<i\leq r$.

\begin{definition}[HOCBF \cite{c2}]\label{def:HOCBF}
The function $h$ in \eqref{BF} is a \emph{High Order Control Barrier Function} (HOCBF) of relative degree $r$ for system \eqref{LTI} if there exist differentiable class $\mathcal{K}$ functions $\alpha_1, \alpha_2, \dots, \alpha_r$, such that
\begin{equation} \label{HOCBF}
\psi_{r}({x}) \geq 0
\end{equation}
for all ${x} \in \hat{C} = \bigcap_{i=1}^r C_{r}$.
\end{definition}

\begin{property}
 Given an HOCBF $h$, if $x(t_0)\in \hat{C}$, then any any Lipschitz
continuous controller $u(t)$ satisfying (\ref{HOCBF}) renders the set $\hat{C}$ forward invariant for system (\ref{LTI}).
\end{property}
See \cite{c2} for a proof. Note that for $r=1$ we recover the original definition of a Zeroing Control Barrier Function (ZCBF) \cite{c1}. Moreover, for the particular case where the functions $\{\alpha_i\}_{i=1}^r$ are linear scalar functions (with positive coefficients), Definition~\ref{def:HOCBF} simplifies to the definition of ECBF from \cite{c3}:
\begin{definition}[ECBF, \cite{c3}]: The function $h$ in \eqref{BF} is an \emph{Exponential Control Barrier Function} (ECBF) of relative degree $r$ for system \eqref{LTI} if there exist coefficients $a_h\in\real{1\times r}$ satisfying the stability conditions in \cite{c3} and such that
\begin{equation} \label{ECBF}
L_{f}^{r}h(x)+L_{g} L_{f}^{r-1} h(x) u+\alpha_h \xi_{h}(x) \ge 0
\end{equation}
for all $x\in C$, where 
\begin{equation}
\xi_{h}(x)=\bmat{
h(x) \\
\dot{h}(x) \\
\ddot{h}(x) \\
\vdots \\
h^{\left(r-1\right)}(x)
}=\bmat{
h(x) \\
L_{f} h(x) \\
L_{f}^{2} h(x) \\
\vdots \\
L_{f}^{r-1} h(x)
}
\end{equation}
\end{definition}

In our paper we make use of the simplified ECBF formulation. Moreover, we will use multiple affine barrier functions $h_i(x)=A_{xi}x+b_{xi}$, where $A_{hi}$ and $b_{xi}$ are the individual rows and elements of $A_x$ and $b_x$, respectively. Together, these functions delimit the set $\cX$. Similarly, we will use barrier functions $h_j(x)=A_{uj}x+b_{uj}$ for the set $\cU$.

\subsection{Control Lyapunov Function and Stability Constraints}
Results similar to those of the section above can be obtained for analyzing stability instead of safety. In particular, we will use the notion of  Exponential Control Lyapunov Function (ECLF) from \cite{c6}, which extends similar notions from \cite{c4,c5} to systems with high relative degree.

\begin{definition}[ECLF, \cite{c6}]
  A sufficiently smooth function $V(x): \mathcal{X} \rightarrow \mathbb{R}$ with $V(x) \ge 0$ with relative degree $r \ge 0$ for the dynamics \eqref{LTI} is an Exponential Control Lyapunov Function (ECLF) if there exists $\beta_v \in \mathbb{R}^{1\times r}$ satisfying the stability requirements given in \cite{c6} and control inputs $u$ such that 

\begin{equation} \label{CLF}
L_{f}^{r} V(x)+L_{g} L_{f}^{r-1} V(x) u+\beta_{V} \xi_{V}(x) \leq 0
\end{equation}
for all $x\in \mathcal{X}$, where
\begin{equation}
\xi_{V}=\bmat{
V(x) \\
\dot{V}(x) \\
\vdots \\
V^{(r-1)}(x)
}=\bmat{
V(x) \\
{L}_{f} V(x) \\
\vdots \\
{L}_{f}^{r-1} V(x)
}
\end{equation}
\end{definition}

\begin{property}
  Given an ECLF $V(x)$, if the set $\mathcal{X}$ is a forward-invariant set, then $\lim _{t \rightarrow \infty} V(x(t))=0$ with exponential convergence. Furthermore, if $\alpha_{1}(\norm{x}) \leq V(x) \leq \alpha_{2}(\norm{x})$ for all $x \in \mathcal{X}$, and class-$\mathcal{K}$ functions $\alpha_1$ and $\alpha_2$, then $\lim _{t \rightarrow \infty} x=0$ with exponential convergence.
\end{property}
See \cite{c6} for the proof.

\subsection{Convex Approximation for Chance Constraints}\label{sec:chance-relaxation}
A generic chance-constrained optimization problem has the following form
\begin{subequations}
\begin{align}\label{eq:generic-chance-constrained}
\min_{x\in \cX} &\quad \varphi_0(x)\\
s.t. \; &\cP(\varphi_i(x) \geq 0) \leq \eta_i(x), \quad i=1,\dots,n,\label{eq:chance-constraint}
\end{align}
\end{subequations}
where $\varphi_i$ denote random variables that depend on the state $x$, and $\eta_i(x)\in [0,1]$ are user-supplied chance constraints which, in general, are functions of the state $x$. We stress the fact that $x$ is deterministic, while the effect of noise is captured by $\varphi_i$.
In general, the constraint \eqref{eq:chance-constraint} is non-convex, but we can replace the inequality with a sufficient condition to obtain a convex relaxation. For this, we need the following proposition (inspired by \cite{c8}):
\begin{proposition}
  For any non-negative, convex, and non-decreasing function $\phi(u)$ satisfying $\phi(0)=1$ and for any $t>0$, the constraint \eqref{eq:chance-constraint} is implied by the condition
  \begin{equation} \label{chance3}
    \expect[\phi(\frac{\varphi_i(x)}{t})] \le \eta_i
  \end{equation}
\end{proposition}
\begin{proof}
  Introducing the indicator function
  \begin{equation}
    I_+(u)=\begin{cases}
      1  & {u \geq 0},\\
      0  & {u < 0},
    \end{cases}
  \end{equation}
  we can express the probability in the chance constraints as an expectation: $\cP(\varphi_i(x) \geq 0) = \mathbf{E}[I_+(\varphi_i(x))]$.
   Together with the fact that $\phi(u)\ge I_+(u)$, we have
\begin{equation} \label{chance2}
  \cP(\varphi_i(x)\ge 0) =\cP(\frac{\varphi_i(x)}{t}\ge 0) 
\le \expect[\phi(\frac{\varphi_i(x)}{t})].
\end{equation}
The claim follows.
\end{proof}

Specifically, let $\varphi_i = b\transpose(x+\theta)+c$, where $a\in \mathbb{R}$, $b\in \mathbb{R}^{n}$, and $\theta \in\real{n}$ is a random variable with $\expect[\theta] = 0$, $\expect[\theta\theta\transpose]=\Sigma(x)$, where $\Sigma(x)$ is a \emph{linear}  function of $x$. By choosing $\phi(u) = (u+1)^2$ (Chebyshev bound), we have:

\begin{multline}
\cP(f \ge 0) \le \eta \Longleftarrow \expect[(f + t)^2]\le t^2\eta\\
\iff \expect[f]^2 + 2t\expect[f]+t^2(1-\eta) \le 0\\
\iff b\transpose(xx\transpose+\Sigma(x))b+2cb\transpose x+c^2\\ +2t(b\transpose x+c)+t^2(1-\eta) \le 0.
\end{multline}
The final expression is equivalent to 
\begin{equation} \label{eq:chebychev}
(b\transpose x+c)^2+b\transpose\Sigma(x) b\transpose +2t(b\transpose x+c)+t^2(1-\eta) \le 0
\end{equation}

For any given $t$, \eqref{eq:chebychev} defines a quadratic convex relaxation of \eqref{eq:chance-constraint}. The constant $t > 0$ is assumed to be defined by the user (in future work we will explore the possibility of including it in the optimization problem). 

\begin{remark}
  If $\Sigma(x)$ is a \emph{piecewise linear} function of $x$, we can partition the original domain $\cX$ in convex cells such that $\Sigma(x)$ is linear on each element of the partition.
\end{remark}

\subsection{Robust Optimization}\label{sec:robust-optimization}

Consider the following (infinite) set of parametrized constraints:
\begin{equation}\label{QC}
x^{T} A^{T} A x+b^{T} x+c \leq 0\;\forall x, \quad A_xx \le b_x.
\end{equation}

This is a robust optimization problem \cite{c7}, and is equivalent to:
\begin{equation} \label{f14}
\sup\limits_{A_xx \le b_x}(x^{T} A^{T} A x+b^{T} x)\leq -c
\end{equation}
The left side of \eqref{f14} is a non-convex problem. However, because the target function is second order differentiable and convex (quadratic), we have the following:
\begin{proposition}[Vertex trick]\label{prop:vertex-trick}
  The optimal value of \eqref{f14} locates at the vertex of the polytope $A_xx \le b_x$, and condition \eqref{f14} is equivalent to 
  \begin{equation} \label{robust1}
    x_i^{T} A^{T} A x_i+b^{T} x_i \le -c, \quad i=1,2,\dots, n{}\\
  \end{equation}
where $x_1,\dots,x_n$ are all vertexes of the polygon $A_xx \le b_x$
\end{proposition}
\begin{proof}
  The target function is convex (the problem is nonconvex). By way of contradiction, suppose that the maximum is achieved at a point $x_{max}$ that is in the interior of the polygon. Since the optimization objective is second order differentiable, its second order derivative (Hessian) must be negative definite in a neighborhood of $x_{max}$. This contradicts the fact that the optimization objective is defined to be convex. A similar argument can be made for $x_{max}$ belonging to a face of the polytope (without it being a vertex), after restricting the objective function to the linear subspace containing that face.
\end{proof}
In other words, the infinite number of constraints \eqref{QC} can be reduced to the $n$ constraints in \eqref{robust1}, where $n$ is the number of vertices of the polytope.

\section{Chance Constraint Robust Control} \label{section3}
As anticipated in Section~\ref{sec:problem_statement}, we consider a linear controller of the form $u=K(x+\theta)$, where $x+\theta$ represents a noisy measurement of the agent's state. As a result, expressions such as the CBF constraint \eqref{ECBF} and the CLF constraint \eqref{CLF} become stochastic quantities. 
In this section, we introduce the probabilistic version of CBF constraint which can make sure the forward-invariant of the set in the sense of probability. We present a general method to solve the optimization problem with chance CBF constraint. To demonstrate this approach, we show how to design a linear state feedback controller.

\subsection{Set Invariance and Chance Constraint} \label{section3a}
The most natural way to transform \eqref{ECBF} into a deterministic quantity is to reformulate it using chance constraints:
\begin{equation} \label{CCBF} 
  \cP\bigl( L_{f}^{r}h(x)+L_{g} L_{f}^{r-1} h(x) u+\alpha_h \xi_{h}(x) \ge 0 \bigr) \ge 1-\eta(x)
\end{equation}
where, as before, $\eta(x)\in[0,1]$. We refer to \eqref{CCBF} as a \emph{Chance CBF (CCBF) constraint}. 

\begin{proposition}\label{prop:forward-invariance}
Assume that $\eta(x) = 0$  for all $x \in \partial C$ and that there exist $K$ such that \eqref{CCBF} is satisfied for all $x\in C$ under the dynamics \eqref{LTI} and the controller \eqref{eq:controller}. Then, $\cC$ is forward invariant almost surely (i.e., if $x(0)\in\cC$, then $x(t)\in\cC$ for all $t$ and all the realizations of $\theta$).
\end{proposition}
\begin{proof}
Let $r$ be the relative degree of \eqref{LTI} with respect to $h$. This means that the only term in \eqref{eq:psi} containing the control input $u$ is $\psi_{r}$. Based on \cite[Theorem 4]{c3}, it is sufficient to show that $\psi_{r-1} > 0$ almost surely, i.e., that the set $C_r$ is forward invariant with probability one. By way of contradiction, if $C_r$ is not forward invariant, then $\exists x \in \partial C_r$ such that $\dot \psi_{r-1} < 0$ for some realization of $\theta$ having non-zero probability. However, from \eqref{CCBF}, for $\forall x \in \partial C_r$, we have $\cP(\dot \psi_{r-1}\ge 0) = 1$, hence $\cP(\dot \psi_{r-1} < 0)=0$, leading to a contradiction.
\end{proof}
\begin{remark}\label{remark:forward-invariance}
  Note that the conditions for almost sure forward invariance in Proposition~\ref{prop:forward-invariance} imply a close relation between $\Sigma(x)$ and $\eta(x)$. In particular, from the proof above, the condition $\cP(\dot\psi_{r-1} > 0)=1$ implies strong constraints on the possible realization of $\theta$, and therefore on the practical sensors used by the robot. For instance, intuitively, when $x$ is in correspondence of a wall, the measurements of distances from the wall should be non-negative; in practice this could be implemented with additional sensors (e.g., short-range proximity sensors to complement long-range Lidar).
\end{remark}
The probability $\eta(x)$ is a function of $x$. For the convenience of solving the optimization problem, the $\eta(x)$ should have the following properties:
\begin{itemize}
\item $\eta(x)$ is concave;
\item $\eta(x) = 0$ for $x \in \partial C$;
\item $0<\eta(x)<1$ for $x \in C$.
\end{itemize}

In this paper we use
\begin{equation} \label{eta}
\eta(x) = \log \left(\frac{d(x)}{\gamma}+1 \right)
\end{equation}
where $d(x)$ is the distance from $x$ to the boundary (see Appendix-\ref{apdx1}), $\gamma$ is number that depends on the size of the environment $\max_{x\in\cX} d(x)$ such that $\eta(x)\in[0,1]$ for $x \in C$.

Similarly to the CCBF constraints, we introduce the \emph{Chance CLF (CCLF) constraints} based on the CLF constraint \eqref{CLF}
\begin{equation} \label{CCLF}
\cP(L_{f}^{r} V(x)+L_{g} L_{f}^{r-1} V(x) u+\beta_{V} \xi_{V}(x) \leq 0) \ge 1 - \eta_v
\end{equation}
and the \emph{Chance Actuator (CACT) constraint}
\begin{equation} \label{CACT}
\cP(u \in \cU) \ge 1 - \eta_u
\end{equation}
where $\eta_c,\eta_u\in [0,1]$ are small constant probabilities.

\subsection{Controller with Chance Constraints} \label{section3b}
In this section we address the control objectives in Problem~\ref{def:problem}. This mainly consists of defining an appropriate Control Lyapunov Function and formulating the optimization programs for finding the matrix $K$ for the controller \eqref{eq:controller}.


\begin{figure}[bt]
  \centering
  \includegraphics[scale=0.25]{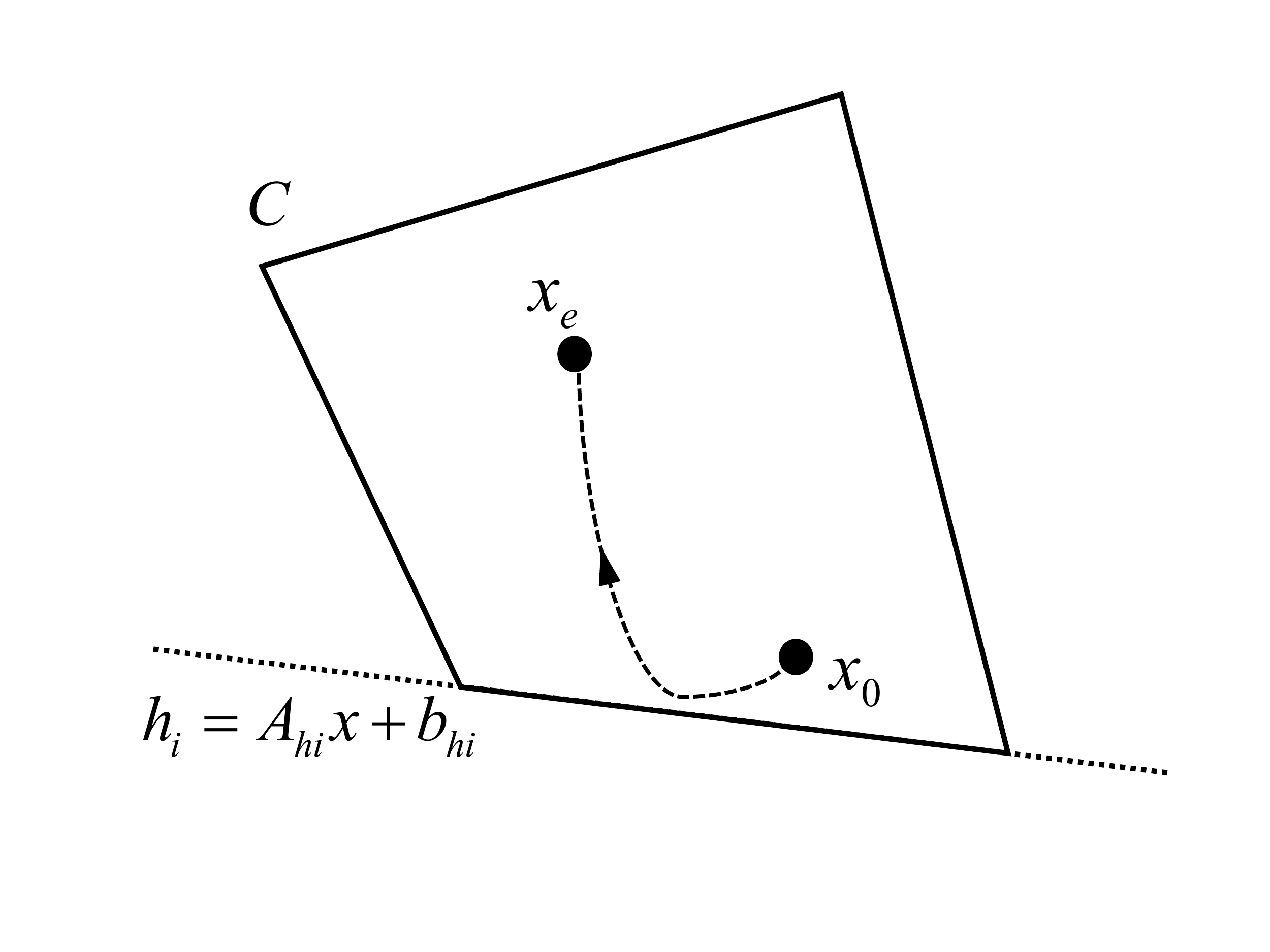}
  \caption{The system running in the convex cell $C$ without touching the boundary. $x_0$ is the initial point. $x_e$ is the end point.}\label{fig:1}
\end{figure}

\subsubsection{Equilibrium Control \ref{it:equilibrium}}
Without loss of generality, we assume that the desired equilibrium is at the origin $x_e=0$ (if not, we can perform a simple change of coordinates; note that the dynamics \eqref{LTI} does not change).
Our goal is to obtain a controller matrix $K$ such that the trajectories of the system converge to a neighborhood of $x_e$ (with exact convergence to $x_e$ when $\theta=0$); see Fig.~\ref{fig:1} for an illustration of this problem.

We select a Lyapunov function candidate of the form
\begin{equation}
V(x) = x\transpose Px,
\end{equation}
where $P$ is a positive definite matrix to be designed.
Following the definitions of Section~\ref{section2}, the necessary Lie derivatives can be computed as
\begin{align} \label{env1-setup}
L^r_{f}h_i(x)&=A_{xi}A^rx,\\
{L}_g{L}^{r-1}_{f}h_i(x)&=A_{xi}A^{r-1}B,\\
{L}_{f}V(x)&=x^T(A^TP+PA)x,\\
{L}_gV(x)&=2x^TPB.
\end{align}

Combining the CCBF \eqref{CCBF}, CCLF \eqref{CCLF}, and CACT \eqref{CACT} constraints with the dynamics \eqref{LTI}, we propose to find $K$ via the following optimization problem
\begin{align}
&\underset{K,Q}{\operatorname{min}} \quad \norm{K}^2_F   \label{q1} \\
&s.t. \notag \\
&\left\{
\begin{array}{l@{}l@{}l@{}l}
& \cP\{\mathcal{A}_ix + A_{xi}A^{r-1}BK(x+\theta) +\alpha_{hi,0}b_{hi}\ge 0 \}\\
& \qquad \qquad \qquad \qquad \qquad \qquad \qquad \qquad \ge 1-\eta_{hi}(x)\\
&\cP\{-x^TQx + 2x^TP BK\theta \le 0 \}\ge 1-\eta_v\\
& \cP\{A_{uj}K(x+\theta)\le b_{uj} \}\ge 1-\eta_{uj}\\
&-Q = (A+BK)^TP+P(A+BK)+\beta_V P\\
& Q\succ 0\\
& i=1,\dots,n_h\\
& j=1,\dots,n_u\\
& \forall x: A_xx\le b_x
\end{array} \right. \notag
\end{align}
where $\mathcal{A}_i = A_{xi}A^r + \sum^{r-1}_{m=0}\alpha_{hi,m}A_{xi}A^m$,
$\alpha_{hi,m}$ is the $(m+1)^{th}$ entry of the vector $\alpha_{hi}$, $A_x = -A_h$ and $b_x = b_h$.

\begin{remark}
If we included $P$ together with $K$ and $Q$ as variables in the optimization problem \eqref{q1}, we would have obtained bilinear terms in $P$ and $K$. Instead, we propose to find $P$ for the Lyapunov function by solving a separate optimization problem (see the Appendix \ref{apdx2}).
\end{remark}
\begin{remark}
  The objective function in \eqref{q1} can be selected to be any linear or quadratic function of the optimization variables. We selected $\norm{K}_F^2$ to encourage ``small effort'' controller that tend to produce inputs $u$ with small norm.
\end{remark}

\subsubsection{Path Control \ref{it:path}} Our goal is to obtain a controller matrix $K$ such that the trajectories of the system reach a predetermined exit face of the polyhedron $\cX$ when starting from any point in $\cX$; see \ref{fig:2} for an illustration. Note that when defining the CBF functions $h_i(x)=A_{hi}x+b_{hi}$ for obstacle avoidance (as defined in Section~\ref{sec:CBFintro}), we omit the exit face. 

\begin{figure}[b]
  \centering
  \includegraphics[scale=0.25]{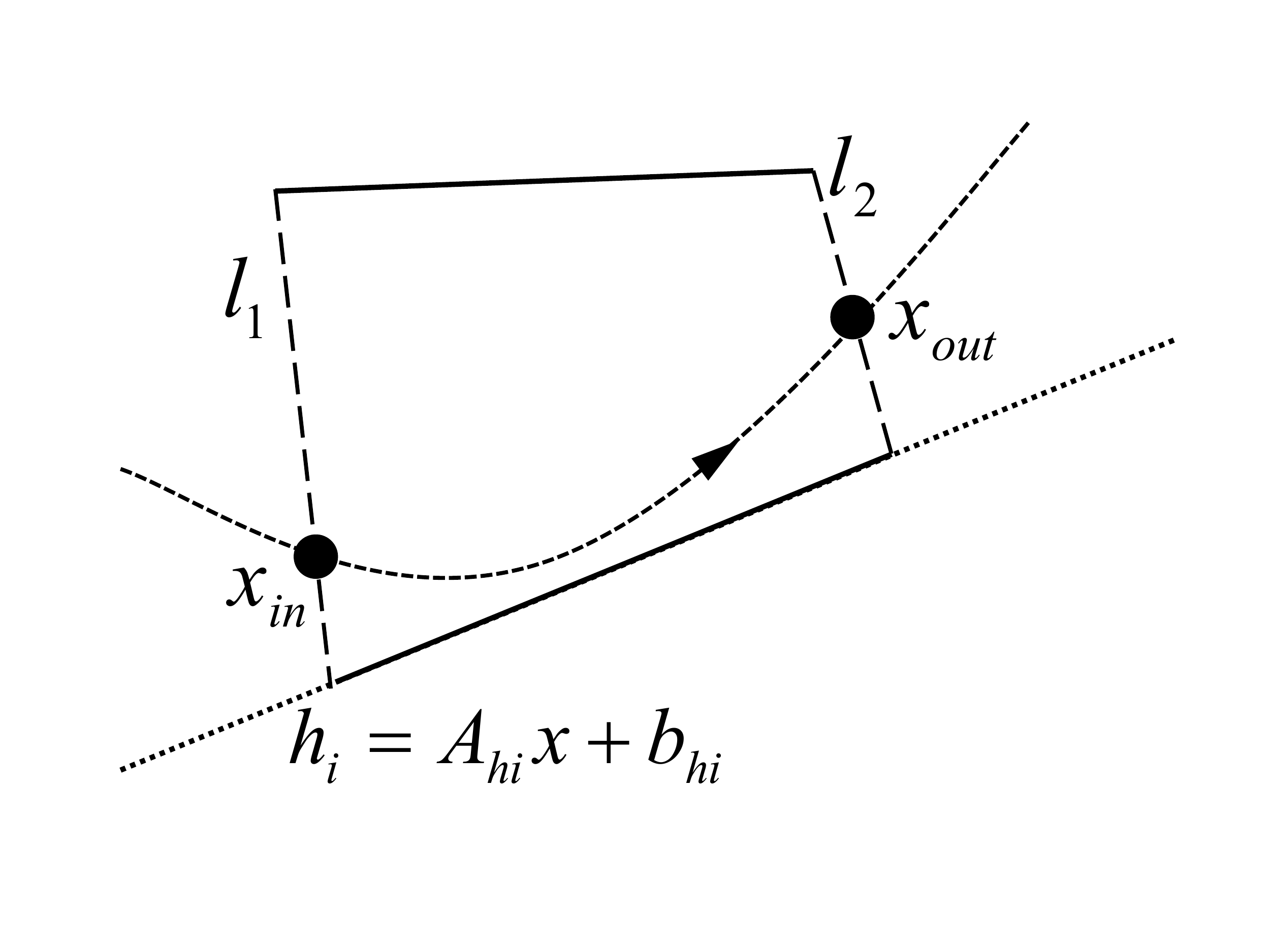}
  \caption{Path Control Task: The system starts from a point $x_{in}$ in the polygon (possibly on a facet $l_1$) and leaves from a point $x_{out}$ on the exit facet $l_2$ without touching the boundary.}\label{fig:2}
\end{figure}

We select a Lyapunov function candidate for the exit face is
\begin{equation}
    V(x) = z^Tx+b_z
\end{equation}
where $z$ andn $b_z$ are defined such that $V(x)=0$ for all $x$ in the exit face (i.e., $z$ and $b_z$ represent the normal and distance of the hyperplane containing the exit face), and $V(x)>0$ inside the polygon.

Following the definitions of Section~\ref{section2}, the necessary Lie derivatives can be computed as
\begin{align} \label{env2-setup}
L^r_{f}h_i(x)&=A_{xi}A^rx,\\
{L}_g{L}^{r-1}_{f}h_i(x)&=A_{xi}A^{r-1}B,\\
{L}^r_{f}V(x)&=z^TA^rx,\\
{L}_g{L}^{r-1}_{f}V(x)&=z^TA^{r-1}B.
\end{align}

Combining the CCBF \eqref{CCBF}, CCLF \eqref{CCLF}, and CACT \eqref{CACT} constraints with the dynamics \eqref{LTI}, we propose to find $K$ via the following optimization problem.
\begin{align}
&\underset{K}{\operatorname{min}} \quad \norm{K}_F^2   \label{q2} \\
&s.t. \notag\\
&\left\{
\begin{array}{l@{}l@{}l@{}l}
& \cP\{\mathcal{A}_ix + A_{xi}A^{r-1}BK(x +\theta) +\alpha_{hi,0}b_{hi} \ge 0 \}\\
& \qquad\qquad\qquad\qquad\qquad\qquad\qquad\qquad \ge 1-\eta_{hi}(x)\\
&\cP\{\mathcal{A}_zx + z^TA^{r-1}BK(x  +\theta) +\beta_{V,0}b_{z}\le 0 \}\ge 1-\eta_v\\
& \cP\{A_{uj}K(x+\theta)\le b_{uj} \}\ge 1-\eta_{uj}\\
& i=1,\dots,n_h-2\\
& j=1,\dots,n_u\\
&\forall x: A_xx \leq b_x
\end{array} \right. \notag
\end{align}
where $\mathcal{A}_z = z^TA^r+\sum^{r-1}_{m=0}\beta_{V,m}z^TA^m$ and $\beta_{V,m}$ is the $(m+1)^{th}$ entry in the vector $\beta_V$.

\subsection{Convex Approximation and Robust Optimization} \label{section4}
In general, both optimization problems \eqref{q1} and \eqref{q2} are nonconvex, and contain an infinite number of constraints due to the expression $\forall x:A_xx\leq b$. We propose to find convex relaxations of these problems by solving two steps:
\begin{enumerate}
\item Apply the convex relaxation of the chance constraints based on the Chebyshev bound (Section~\ref{sec:chance-relaxation}).
\item Reduce the infinite number of constraints to a finite set given by the vertices of the polyhedron $\cX$ (Section~\ref{sec:robust-optimization}).
\end{enumerate}



Let us consider the chance CBF constraint (\ref{q2}) as an example. For the sake of simplicity, let $K_i = A_{xi}A^{r-1}BK$, $K'_i= \mathcal{A}_i + K_i$ and $\alpha_{i} = \alpha_{hi,0}b_{hi}$. Following the derivation of \eqref{eq:chebychev}, we obtain:

\begin{equation} \notag
\begin{aligned}
&P\{\mathcal{A}_ix + A_{xi}A^{r-1}BK(x+\theta) +\alpha_{hi,0}b_{hi} \ge 0 \} \\ 
&\qquad\qquad\qquad\qquad\qquad\qquad\qquad\qquad\qquad  \ge 1-\eta_{hi}(x)\\
& \Longleftrightarrow P\{K_i^{\prime}x + K_i\theta +\alpha_{i} \ge 0 \} \ge 1-\eta_{hi}(x)\\
& \Longleftarrow x^TK_{i}^{\prime T}K_{i}^{\prime} x+2(\alpha _i-t_i)K_{i}^{\prime}x - t_i^2\eta_{hi}(x)\\
&\qquad \qquad \qquad \qquad \qquad \qquad +K_i\Sigma K_i^T + (\alpha_{i}-t_i)^2  \le 0
\end{aligned}
\end{equation}

Since the inequality needs to be satisfied for all $x\in\cX$, we can equivalently write 
\begin{equation} \notag
\begin{aligned}
&\sup\limits_{A_xx \le b_x} \left\{ x^TK_{i}^{\prime T}K_{i}^{\prime} x+2(\alpha _i-t_i)K_{i}^{\prime}x - t_i^2\eta_{hi}(x) \right\}\\
&\qquad \qquad \qquad \qquad \qquad \qquad +K_i\Sigma K_i^T + (\alpha_{i}-t_i)^2  \le 0
\end{aligned}
\end{equation}

Applying Proposition~\ref{prop:vertex-trick}, we can use vertices to get the upper bound of the supremum.

\begin{equation} \notag
\begin{aligned}
\max _{x_k\in \{x_1,x_2,\dots,x_{n_p}\}} \left\{
\begin{matrix}
 x_k^{T} K_{i}^{\prime T} K_{i}^{\prime} x_k+2\left(\alpha_{i}-t_{i}\right) K_{i}^{\prime} x_k\\
-t_i^{2}\eta_{h i}(x_k)
\end{matrix}
\right\} \\+K_i\Sigma K_i^T + (\alpha_{i}-t_i)^2  \le 0\\
\Longleftrightarrow  x_k^{T} K_{i}^{\prime T} K_{i}^{\prime} x_k+2\left(\alpha_{i}-t_{i}\right) K_{i}^{\prime} x_k
-t_i^{2}\eta_{h i}(x_k)\\
+K_i\Sigma K_i^T + (\alpha_{i}-t_i)^2  \le 0\\
\end{aligned}
\end{equation}
for $k = 1,2,\dots, n_p$, where $n_p$ is the number of vertices, $x_k$ is the $k^{th}$ vertex of the polygon.

Hence, the original CBF constraint for $h_i$ becomes
\begin{equation} \label{ccbf}
\begin{matrix}
 i^{th}\\\text{CBF}
\end{matrix}
\left\{
\begin{array}{l@{}l@{}l@{}l}
&  K_{i,k}^{\prime T} K'_{i,k}+2\left(\alpha_{i}-t_{i}\right) K_{i}^{\prime} x_k
-t_i^{2}\eta_{h i}(x_k) \\
&\qquad\qquad\qquad\qquad+K_i\Sigma K_i^T + (\alpha_{i}-t_i)^2  \le 0\\
& K_i =A_{xi}A^{r-1}BK\\
& K'_i=\mathcal{A}_i + K_i\\
& K'_{i,k} = K'_ix_k\\
& k=1,\dots,n_p
\end{array}
\right.
\end{equation}
where $\alpha_{i} = \alpha_{hi,0}b_{hi}$ and $t_i$ is the parameter for convex approximation (\ref{chance3}). This is a quadratic constraint.

Similarly, for the chance CLF constraint in (\ref{q1}), we get
\begin{equation} \label{cclf}
\text{CLF}
\left\{\begin{array}{l@{}l@{}l@{}l}
&  Q_k^Tx_kx_k^TQ_k + 4K_{v,k}^T\Sigma K_{v,k}+t_v^2(1-\eta_v) \le0\\
&-Q = (A+BK)^TP+P(A+BK)+\beta_V P\\
& K_v=PBK\\
& Q_k = Qx_k \\
& K_{v,k} = K^T_vx_k\\
& Q \succ 0 \\
& k=1,\dots,n_p
\end{array}
\right.
\end{equation}
where $t_v$ is the parameter for convex approximation. 

For the chance CLF constraint in (\ref{q2}), we get
\begin{equation} \label{cclf2}
\text{CLF}
\left\{
\begin{array}{l@{}l@{}l@{}l}
&  K_{z,k}^{\prime T} K'_{z,k}+2\left(\beta_{z}+t_{z}\right) K_{z}^{\prime} x_k +K_z\Sigma K_z^T\\
&\qquad\qquad\qquad\quad + (\beta_{z}+t_z)^2 -t_z^{2}\eta_{v}  \le 0\\
& K_z =z^TA^{r-1}BK\\
& K'_z=\mathcal{A}_z + K_z\\
& K'_{z,k} = K'_zx_k\\
& k=1,\dots,n_p
\end{array}
\right.
\end{equation}
where $\beta_{z} = \beta_{z,0}b_{z}$ and $t_z$ is the parameter for convex approximation.

For the chance actuator constraint in (\ref{q1}) and (\ref{q2}), we get
\begin{equation} \label{cact}
\begin{matrix}
 j^{th}\\\text{ACT}
\end{matrix}
\left\{
\begin{array}{l@{}l@{}l@{}l}
&  K_{u,k}^TA^T_{uj}A_{uj}K_{u,k}+2(t_{uj}-b_{uj})A_{uj}K_{u,k}\\
& \qquad +K_{uj}\Sigma K_{uj}^T+(t_{uj}-b_{uj})^2 - t_{uj}^2\eta_{uj}\le 0\\
& K_{u,k} = Kx_k\\
& K_{uj} = A_{uj}K\\
& k=1,\dots,n_p
\end{array}
\right.
\end{equation}

Finally, a feasible (but conservative) solution for (\ref{q1}) and (\ref{q2}) can be found by solving the following optimization problem

\begin{align} \label{result}
&\underset{K}{\operatorname{min}} \quad \norm{K}_F^2 \\
&s.t. \notag \\
&\left\{
\begin{array}{l@{}l@{}l@{}l}
& \text{$i^{th}$ CBF constraint (\ref{ccbf})}\\
& \text{CLF constraint (\ref{cclf}) or (\ref{cclf2})} \\
& \text{$j^{th}$ Actuator  constraint (\ref{cact})}\\
& i=1,\dots,n_h\\
& j=1,\dots,n_u\\
\end{array} \right. \notag
\end{align}
where $K$ is the feedback matrix for the controller. All other variables having capital $K$ are intermediate decision variables.  

This optimization problem contains quadratic inequality constraints and linear equality constraints. It is therefore a Quadratic Constraint Quadratic programming (QCQP).

\begin{figure*}[t] 
  \centering
  \subfloat[$\sigma = 0$]{
    \includegraphics[width=2in]{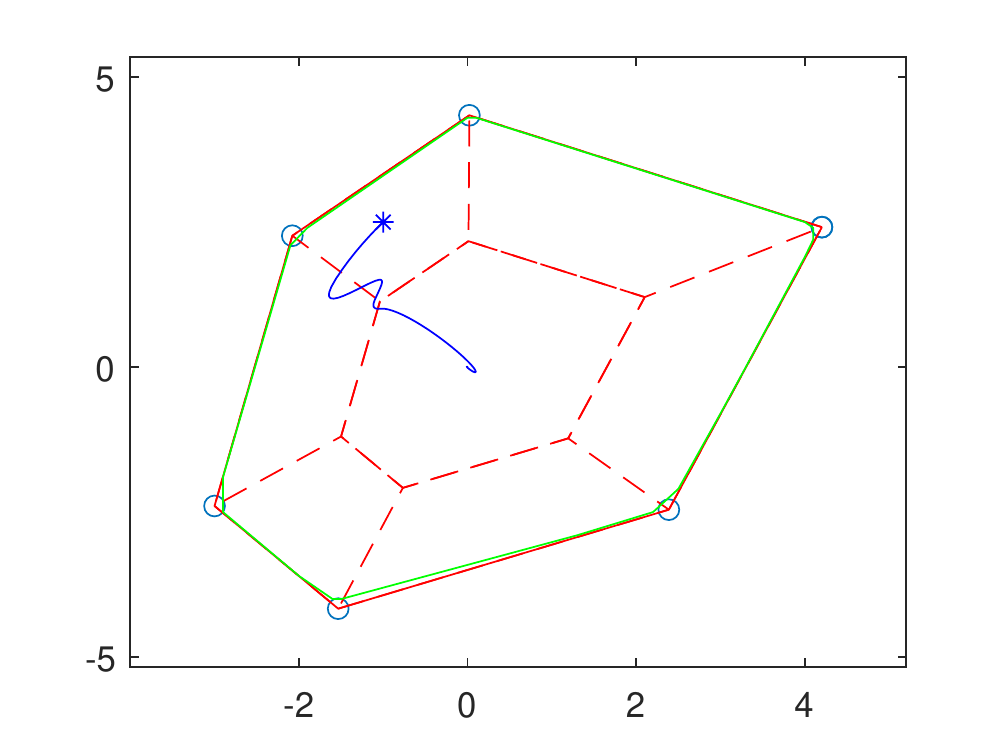}
  }%
  \subfloat[$\sigma = 0.1$]{
    \includegraphics[width=2in]{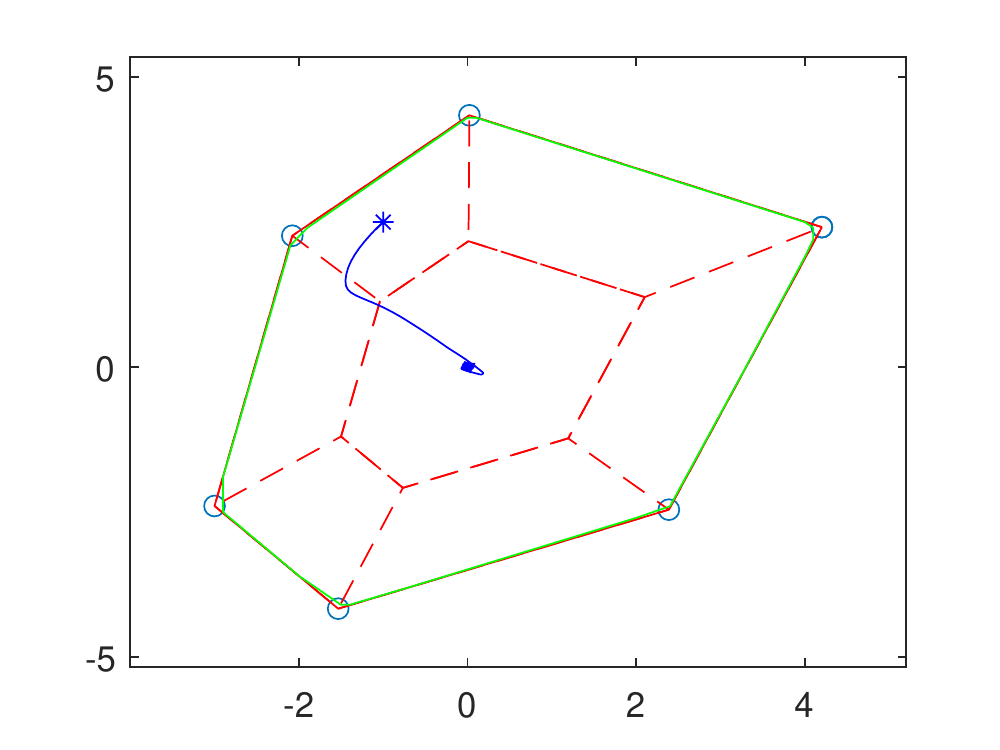}
  }%
  \subfloat[$\sigma = 1$]{
    \includegraphics[width=2                                  in]{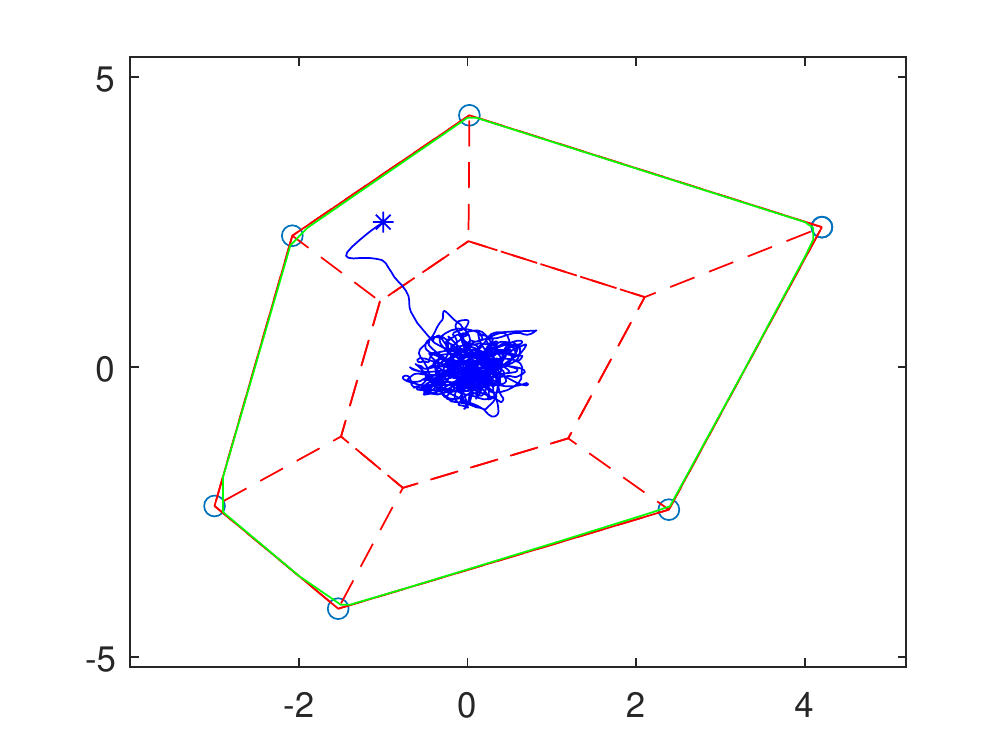}
  }%
  \caption{Equilibrium Control: Running path and invariant set for two-dimension second order integrator with different noise}\label{fig:3}
\end{figure*}
\begin{figure*}[t] 
  \centering
  \subfloat[$\sigma = 0$]{
    \includegraphics[width=2in]{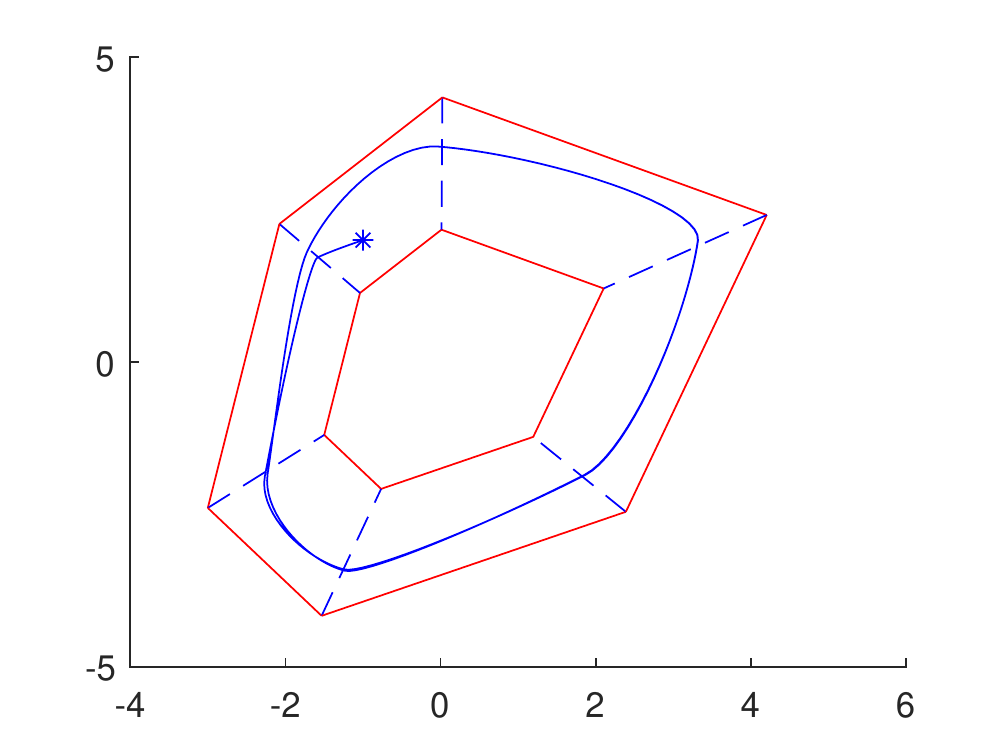}
  }%
  \subfloat[$\sigma = 0.1$]{
    \includegraphics[width=2in]{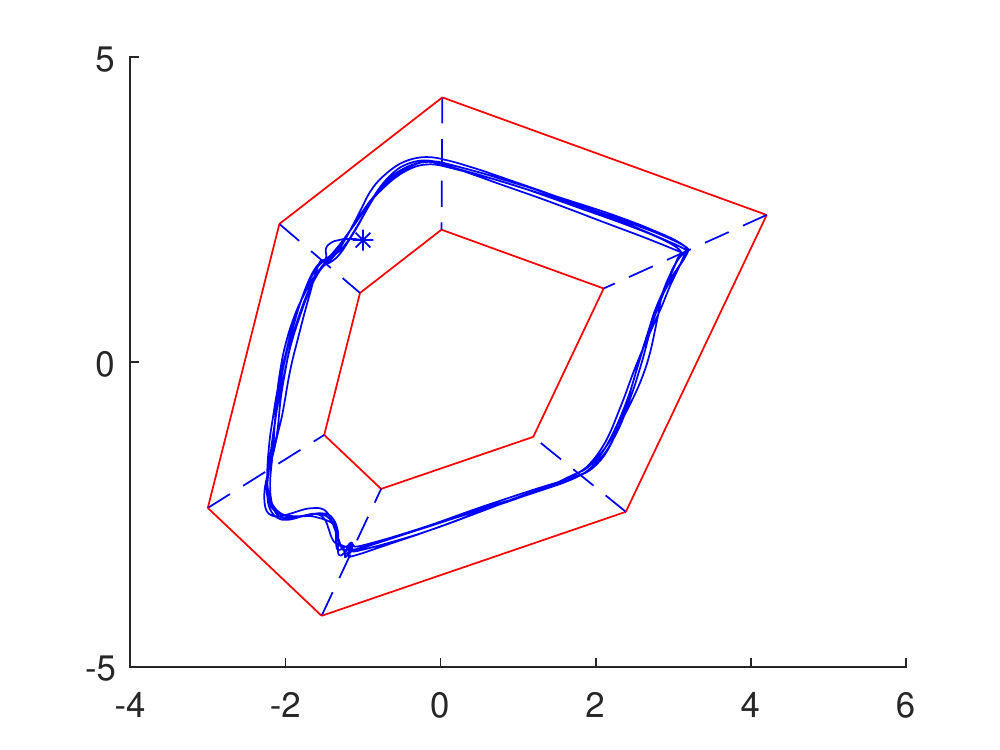}
  }%
  \subfloat[$\sigma = 1$]{
    \includegraphics[width=2in]{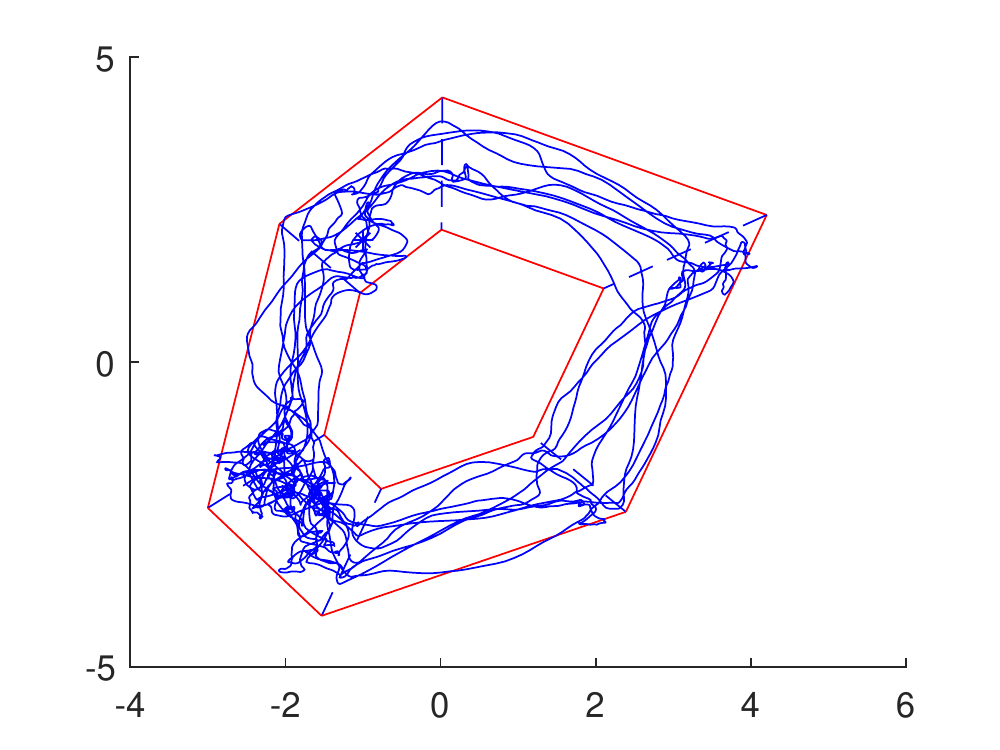}
  }%
  \caption{Path control: Running path for two-dimension second order integrator with different noise}\label{fig:4}
\end{figure*}

\section{Numerical Experiment } \label{section5}
In this section, we do simulations to show the effectiveness of the controller given by (\ref{result}). In order to show the robustness of our approach to non-ideal sensors, we consider a simple case where the covariance matrix of the noise is constant (i.e., it does not depend on $x$). This means that the conditions discussed in Remark~\ref{remark:forward-invariance} are not satisfied. As such, it is impossible to find a feasible solution satisfying all constraints in (\ref{result}) strictly; in the simulations below we show that we can modify \eqref{result} by introducing slack variables to make the problem feasible, while still obtaining a controller with good empirical behavior.

In this section, we use the second-order integrator for the dynamics \eqref{LTI}, but we place restrictions only on the positions. More precisely, selecting the system variable as $[x, \dot{x}, y, \dot{y}]\transpose$, the barrier function (\ref{BF}) has the following form
\begin{equation} 
  h_i(x) = \bmat{ A_{hi1}& 0 & A_{hi2} & 0}
  \bmat{ x\\ \dot{x} \\ y\\ \dot{y}} + b_{hi}
\end{equation}

The following two simulations share the same parameter setting. All parameters mentioned in previous section including $t, \alpha_{hi,0}$, and $\beta $ are 1. And the failure probability for chance CLF constraint $\eta_v$ is 0.2.

\subsection{Equilibrium Control, Relative degree $r=2$}
Figure \ref{fig:3} gives the simulation result for a two-dimension second-order integrator under different noise. In the figure, the blue curve is the running trace with initial state $x_0 = [-1.2, 0.3, 2.5, 0.5]^T$. The red polygon is the environment boundary. The green polygon is the calculated invariant set with zero initial velocity. The results show that the invariant set almost identifies with the polygon environment except for some corners of the polygon. This means that the system can converge to the equilibrium point without violating the boundary.

\subsection{Path Control, Relative degree $r=2$}
Figure \ref{fig:4} gives the simulation result of the path control task for a two-dimension second-order integrator under different noise. In the figure, the blue curve is the system's running trace. The red polygon is the environment boundary. Two polygons setup a ring environment for the path control task. 

The result shows that the controller works very well in the environment with zero noise or small noise, with the system running inside the given zone and follow the path defined by the CLF. In the environment with strong noise, the controller also works but not perfectly, with occasional violations of the boundary. This is due to the presence of slack variables to make the problem always feasible.

\section{Conclusions} \label{section6}
In this paper, we introduce the chance CBF and prove that system satisfying chance CBF constraint would be forward invariant in the sense of probability. Based on the chance CBF constraint, we proposed a novel approach to find the controller in the given polygon environment with noise. The controller design problem is set up as a chance constraint optimization which is nonconvex. We use convex approximation and the vertex trick to constrain the upper bound of the probability instead of constraining the probability. This leads to a quadratic constraint quadratic programming (QCQP) and allows as to find a feasible solution for the original chance constraint optimization problem. The solution of this QCQP gives a linear feedback controller which is easy to implement.

We validate our approach in two tasks, equilibrium control, and path control. In the equilibrium control task, the system needs to converge to the equilibrium point without violating the given boundary. In the path control task, the system works in the given environment following the path defined by the CLF without violating the boundary. The controller works very well in the equilibrium control task with zero noise, small noise, and large noise. It also works well in the path control task with zero noise and small noise. These results show that the controller is robust to the input.

In future work, we will pursue a better theoretical characterization of the CLF stability guarantees (mirroring what is done in this paper for the CBF safety guarantees), and a more rigorous formulation of relaxations.



\section*{APPENDIX} \label{appendix}
\subsection{Distance to boundary} \label{apdx1}
\begin{figure}[b]
  \centering
  \includegraphics[scale=0.2]{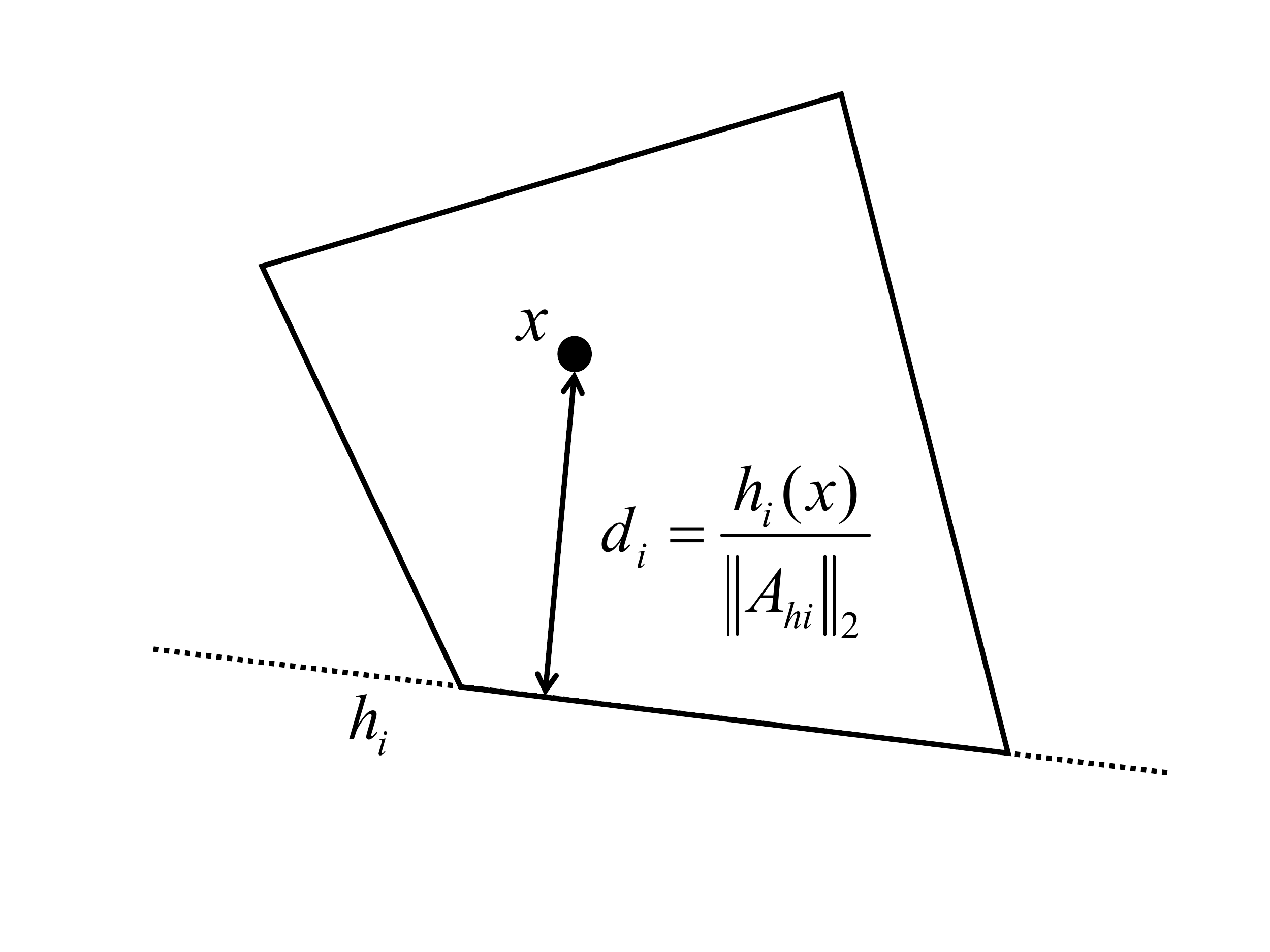}\label{fig:apdx1}
  \caption{$d_i(x)$ is the distance from $x$ to the $i^{th}$ boundary in the polygon}
\end{figure}

\begin{equation}
\begin{aligned} 
d_i &= \frac{h_i(x)}{\Vert A_{xi} \Vert_2} =\frac{A_{xi}}{\Vert A_{xi} \Vert_2}x + \frac{b_{hi}}{\Vert A_{xi} \Vert_2}\\
d   &= A_h\Lambda x + \Lambda b_h\\
\end{aligned}
\end{equation}
where $\Lambda = diag\{\frac{1}{\Vert A_{xi} \Vert_2}\}$ is a diagonal matrix.
The metric $d(x)$ has several properties:
\begin{itemize}
\item $d_i(x) > 0$ for $x\in C$;
\item $d_i(x) = 0$ for $x\in \partial C$;
\item $d_i(x)$ is affine.
\end{itemize}

Because of these properties, $d(x)$ does not affect the concavity properties of (\ref{eta}).

\subsection{Searching for $P$} \label{apdx2}
$P$ is the positive definite matrix for Lyapunov function $V(x) = x^TPx$. To simplify the optimization problem (\ref{q1}), we search for $P$ by solving the following optimization problem.

\begin{align}
&\underset{}{\operatorname{find}} \quad P,K    \\
&s.t. \notag \\
&\left\{
\begin{array}{l@{}l@{}l@{}l}
&(A+BK)^TP+P(A+BK)+\beta_V P \prec  0\\
& P \succ 0\\
\end{array} \right. \notag
\end{align}

This is no LMI, to solve this problem, do the variable substitution $Q = P^{-1}$ and $Y=KP^{-1}=KQ$.

\begin{align}
&\underset{Q,Y,\Lambda}{\operatorname{min}} \quad -\lambda^T \Lambda  \\
&s.t. \notag \\
&\left\{
\begin{array}{l@{}l@{}l@{}l}
&QA^T+ AQ +\beta _VQ+ Y^TB^T+BY + diag(\Lambda)  \prec  0\\
& Q \succ 0\\
& \lambda_i \ge 0\\
\end{array} \right. \notag
\end{align}
where $\lambda \in \mathbb{R}^n$ is the weight for different $\Lambda_i$, $\Lambda$ is a diagonal matrix, and we changed the feasibility problem into a minimization problem. After finding $Q$, we set $P = Q^{-1}$; the solution for $K$ is discarded.

\end{document}